\documentclass{amsart}
\usepackage{verbatim}
\usepackage{paralist}
\usepackage{leftidx}
\usepackage{dsfont}
\usepackage{amsthm}
\usepackage{amsmath}
\usepackage{amssymb}
\usepackage{amscd}
\usepackage{graphicx}
\usepackage{setspace}
\usepackage[all]{xy}
\usepackage{mathrsfs} 
\usepackage{hyperref}
\title{On the Additivity of Newton--Okounkov Bodies}
\author{Robert Wilms}
\address{Universit\'e de Caen Normandie, CNRS, LMNO UMR 6139, F-14000 Caen, France}
\email{\href{mailto:robert.wilms@unicaen.fr}{robert.wilms@unicaen.fr}}
\thanks{The author gratefully acknowledges support from the Swiss National Science Foundation grant ``Diophantine Equations: Special Points, Integrality, and Beyond'' (n$^\circ$ 200020\_184623).}
\begin{document}
	\numberwithin{equation}{section}
	\newtheorem{Def}{Definition}
	\numberwithin{Def}{section}
	\newtheorem{Rem}[Def]{Remark}
	\newtheorem{Lem}[Def]{Lemma}
	\newtheorem{Que}[Def]{Question}
	\newtheorem{Cor}[Def]{Corollary}
	\newtheorem{Exam}[Def]{Example}
	\newtheorem{Thm}[Def]{Theorem}
	\newtheorem*{clm}{Claim}
	\newtheorem{Pro}[Def]{Proposition}

\begin{abstract}
	We study the additivity of Newton--Okounkov bodies. Our main result states that on two-dimensional subcones of the ample cone the Newton--Okounkov body associated to an appropriate flag acts additively. We prove this by induction relying on the slice formula for Newton--Okounkov bodies. Moreover, we discuss a necessary condition for the additivity showing that our result is optimal in general situations. As an application, we deduce an inequality between intersection numbers of nef line bundles.
\end{abstract}
\maketitle
\section{Introduction}
For any line bundle $L$ on an irreducible projective variety $X$, the associated Newton--Okounkov body is a convex body $\Delta_{Y_\bullet}(L)$ in $\mathbb{R}^{\dim X}$ encoding essential properties of $L$, for example its volume. This object was introduced independently by Lazarsfeld--Mustaţă \cite{LM09} and by Kaveh--Khovanskii \cite{KK12} based on ideas by Okounkov \cite{Oko96,Oko03}.
It depends on the choice of a flag
\begin{align*}
Y_\bullet\colon \qquad X=Y_0\supsetneq Y_1\supsetneq Y_2\supsetneq \dots\supsetneq Y_d=\{p\}
\end{align*}
of irreducible subvarieties, which are regular at $p$, where $d=\dim X$. Such a flag is also called an admissible flag. In general, the Newton--Okounkov body satisfies the inclusion
\begin{align}\label{equ_inclusion}
\Delta_{Y_\bullet}(L)+\Delta_{Y_\bullet}(M)\subseteq \Delta_{Y_\bullet}(L\otimes M),
\end{align}
where we take the Minkowski sum on the left-hand side. This inclusion makes it possible to deduce the Brunn--Minkowski inequality 
$$\mathrm{vol}(L\otimes M)^{1/d}\ge \mathrm{vol}(L)^{1/d}+\mathrm{vol}(M)^{1/d}$$
for big line bundles from its classical analogue in convex geometry. Unfortunately, inclusion (\ref{equ_inclusion}) is not sufficient to translate more involved inequalities of mixed volumes in convex geometry to the algebraic geometry of line bundles. In general, (\ref{equ_inclusion}) is not an equality. Hence, it is a natural question in which cases it can be an equality.
In this note we address this question in the special case, where the flag $Y_\bullet$ is obtained by divisors numerically equivalent to rational multiples of $L$. More precisely, we make the following definition.

\begin{Def}
	Let $X$ be any irreducible projective variety of dimension $d$, $Y_\bullet$ an admissible flag on $X$, and $L$ any $\mathbb{Q}$-line bundle on $X$. We say that $Y_\bullet$ \emph{corresponds to $L$} if for all $0\le i\le d-2$ the Weil divisor $Y_{i+1}\subseteq Y_{i}$ is a Cartier divisor and there are rational numbers $r_i\in\mathbb{Q}$ such that $r_i\mathcal{O}_{Y_i}(Y_{i+1})\equiv L|_{Y_i}$, where $\equiv$ denotes numerical equivalence.	
\end{Def}

We denote by $N^1(X)$ the group of numerical equivalence classes of line bundles on $X$ and $N^1(X)_{\mathbb{R}}$ for the real vector space induced by $N^1(X)$. We denote by $\mathrm{Amp}(X)$, $\mathrm{Big}(X)$ and $\overline{\mathrm{Eff}}(X)$ the convex cones in $N^1(X)_\mathbb{R}$ of ample, big and pseudo-effective $\mathbb{R}$-line bundles on $X$. Note that $\mathrm{Amp}(X)\subseteq \mathrm{Big}(X)\subseteq \overline{\mathrm{Eff}}(X)$, that $\mathrm{Amp}(X)$ and $\mathrm{Big}(X)$ are open, and that $\overline{\mathrm{Eff}}(X)$ is the closure of $\mathrm{Big}(X)$.

We will always work over an algebraically closed field of any characteristic and we use additive notation for line bundles on $X$.
For any two $\mathbb{R}$-line bundles $L,M$ we denote the convex cone
$$C_L(M)=\{\lambda L+\mu M~|~\lambda\in\mathbb{R},\mu\in\mathbb{R}_{\ge 0}\}\cap \mathrm{Amp}(X).$$
Our main result states that the Newton--Okounkov bodies $\Delta_{Y_\bullet}(\cdot)$ associated to an admissible flag $Y_\bullet$ corresponding to $L$ act additively on $C_L(M)$.
\begin{Thm}\label{thm_main}
	Let $X$ be any irreducible projective variety, $L$ and $M$ any $\mathbb{Q}$-line bundles on $X$, and $Y_\bullet$ an admissible flag on $X$ corresponding to $L$. Then for all $N_1,N_2\in C_L(M)$ we have
	$$\Delta_{Y_\bullet}(N_1+N_2)=\Delta_{Y_\bullet}(N_1)+\Delta_{Y_\bullet}(N_2).$$
\end{Thm}
Note that the coefficient of $M$ for the classes in $C_L(M)$ must be non-negative. Consequently, if $L$ is ample, it necessarily lies on the boundary of $C_L(M)$. The theorem does not generalize to the scenario where $Y_\bullet$ is an admissible flag corresponding to an arbitrary class in $C_L(M)$. Indeed, if $Y_{\bullet}$ corresponds to an ample $\mathbb{Q}$-line bundle $N\in C_L(M)$, then $\Delta_{Y_\bullet}(N)$ is a simplex \cite[Proposition 4]{AKL14}. Thus, $\Delta_{Y_\bullet}(2N)$ is also a simplex. However, if $N\not\equiv M$ and both $N-\epsilon M$ and $N+\epsilon M$ are also ample $\mathbb{Q}$-line bundles, one can verify that $\Delta_{Y_\bullet}(N-\epsilon M)$ and $\Delta_{Y_\bullet}(N+\epsilon M)$ are in general not simplices. Since a simplex cannot be decomposed into a Minkowski sum of two non-simplices, the equality
$$\Delta_{Y_\bullet}(2N)=\Delta_{Y_\bullet}(N-\epsilon M)+\Delta_{Y_\bullet}(N+\epsilon M)$$
does not hold in general.

To prove the theorem, we will use the description of the slices of Newton--Okounkov bodies by Lazarsfeld--Mustaţă \cite{LM09} to construct the slices of the body on the left-hand side as sum of the slices of the bodies on the right-hand side, using induction on the dimension of $X$. This will give the reverse inclusion to (\ref{equ_inclusion}).
To show potentially applications of Theorem \ref{thm_main} let us give some examples of line bundles which admit corresponding admissible flags.

\begin{Exam}
	\begin{enumerate}[(i)]
		\item If $X$ is any irreducible projective variety and $L$ any ample line bundle on $X$, one can always construct an admissible flag $Y_\bullet$ corresponding to $L$ by Bertini's theorem. We will explain this construction in Section \ref{sec_linear-map}.
		\item Let $X$ be any irreducible projective surface and $Y$ any irreducible Cartier divisor on $X$. The flag $\{p\}\subsetneq Y\subsetneq X$ for any point $p$, regular in $Y$ and $X$, is an admissible flag corresponding to the line bundle $\mathcal{O}_X(Y)$.
		\item Let $X$ and $X'$ be two irreducible projective varieties of dimensions $d'=\dim X'$ and $d=\dim X$. We can always choose an admissible flag $Y'_{\bullet}$ on $X'$. Further, let $L$ be a line bundle on $X$ and $Y_\bullet$ an admissible flag corresponding to $L$. Then the flag
		$$X\times X'\supsetneq Y_1\times X'\supsetneq\dots\supsetneq Y_d\times X'\supsetneq Y_d\times Y'_1\supsetneq \dots\supsetneq Y_d\times Y'_{d'}$$
		is an admissible flag corresponding to the line bundle $\mathrm{pr}_1^*L$ on $X\times X'$, where $\mathrm{pr}_1\colon X\times X'\to X$ denotes the projection to the first factor. The values $r_i$ for this flag are the same as for the flag $Y_{\bullet}$ with the additional values $r_{d-1}=\deg(L|_{Y_{d-1}})$ and $r_i=0$ for $d\le i\le d+d'-2$.
		\item To give a more specific example, let $X$ be any irreducible projective curve. The previous two examples show that the line bundle $\mathcal{O}_{X^d}(\Delta_{jk})$ on the self-product $X^d$ associated to the diagonal divisor
		$$\Delta_{jk}=\{(x_1,\dots,x_d)\in X^d~|~x_j=x_k\}$$
		admits a corresponding admissible flag for $j\neq k$. If $j=1$ and $k=2$ one can give a corresponding admissible flag by
		$$X^d\supsetneq \Delta_{12}\times X^{d-2}\supsetneq \{(p,p)\}\times X^{d-2}\supsetneq \{(p,p,p)\}\times X^{d-3}\supsetneq\dots\supsetneq \{(p,\dots,p)\}$$
		for some regular point $p$ of $X$.
	\end{enumerate}
\end{Exam}

Previously, the additivity of Newton--Okounkov bodies has been studied only for special types of varieties. Łuszcz-Świdecka and Schmitz \cite{LS14} and Pokora, Schmitz and Urbinati \cite{PSU15} proved the existence of Minkowski bases for algebraic surfaces with rational polyhedral pseudo-effective cone respectively for smooth projective toric varieties. A Minkowski base is a finite set of pseudo-effective divisors $D_1,\dots,D_r$ such that any pseudo-effective divisor $D$ can be written as $D=\sum_{i=1}^r a_iD_i$ for some $a_i\ge 0$ and it holds $\Delta_{Y_\bullet}(D)=\sum_{i=1}^r a_i\Delta_{Y_\bullet}(D_i)$, where the $\Delta_{Y_\bullet}(D_i)$'s are indecomposable. Kiritchenko \cite{Kir17} studied the additivity of Newton--Okounkov bodies for Bott--Samelson varieties. The question of additivity has also been discussed for string polytopes of spherical varieties 
in relation to topics like toric degenerations by Alexeev and Brion \cite{AB04} and the description of cohomology rings by Kaveh \cite{Kav11}. Note that string polytopes of spherical varieties can be realized as Newton--Okounkov bodies \cite{Kav15}.

We would like to present Theorem \ref{thm_main} in a more systematic way. On $N^1(X)_\mathbb{Q}$ we have a multi-linear and symmetric intersection map
$$N^1(X)_\mathbb{Q}^d\to \mathbb{Q},\qquad (L_1,\dots,L_d)\mapsto L_1\cdot\ldots \cdot L_d$$
associating to a $d$-tuple of $\mathbb{Q}$-line bundles their intersection number.
On the other hand side, let us write $\mathcal{K}_d$ for the space of convex bodies in $\mathbb{R}^d$, which is a convex $\mathbb{R}_{\ge 0}$-cone. The Minkowski sum for convex bodies satisfies the cancellation rule, that is $K_1+K=K_2+K$ implies $K_1=K_2$. Thus, we can associate to $\mathcal{K}_d$ its Grothendieck group $\mathcal{K}_d^{\mathrm{vs}}$ which is the quotient of the set of formal differences $K_1-K_2$ by the relations $(K_1-K_2)\sim (K_3-K_4)$ if $K_1+K_4=K_2+K_3$. It naturally has the structure of an $\mathbb{R}$-vector space. For more details on this construction we refer to \cite{Est08} and for geometric interpretations we refer to \cite{Kho23, Sch18}.

By multi-linearity we can extend the mixed volume to a multi-linear and symmetric function
$$(\mathcal{K}_d^{\mathrm{vs}})^d\to \mathbb{R},\qquad (K_1,\dots,K_d)\mapsto V(K_1,\dots,K_d),$$
which sends any $d$-tuple $(K_1,\dots,K_d)$ of actual convex bodies to their mixed volume $V(K_1,\dots,K_d)$. As a consequence of Theorem \ref{thm_main} we get linear embeddings of two-dimensional subspaces of $N^1(X)_\mathbb{Q}$ into $\mathcal{K}_d^{\mathrm{vs}}$ respecting the above multi-linear forms.
\begin{Cor}\label{cor_main}
	Let $X$ be any irreducible projective variety of dimension $d$ and $U\subseteq N^1(X)_\mathbb{Q}$ a linear subspace of dimension $\dim U=2$ containing an ample class $L\in U$. Then there exists an injective linear map
	$$\Delta\colon U\to \mathcal{K}_d^{\mathrm{vs}},$$
	which is compatible with the intersection products in the sense that
	$$\tfrac{1}{d!} (M_1\cdot\ldots\cdot M_d)=V(\Delta(M_1),\dots,\Delta(M_d))$$
	for all $M_1,\dots,M_d\in U$. Moreover, for any $r\in\mathbb{N}$ and any ample $\mathbb{Q}$-line bundles $M_1,\dots,M_r\in U$ the map $\Delta$ can be chosen, such that $\Delta(M_j)=\Delta_{Y_\bullet}(M_j)$ is the Newton--Okounkov body associated to some fixed admissible flag $Y_\bullet$ for any $j\le r$.
\end{Cor}

The corollary allows us to translate inequalities of mixed volumes of general convex bodies to inequalities of intersection numbers of ample line bundles in the linear subspace $U$. As $\dim U=2$, it is enough to consider general inequalities of mixed volumes of only two independent convex bodies. However, it has been shown by Shepard \cite{She60} that all these inequalities are induced by the Alexandrov--Fenchel inequality \cite{Ale37,Fen36}, whose analogue is already known for the intersection numbers of ample line bundles, see for example \cite[Theorem 6.1]{Cut15}. Hence, it is a natural question, in which cases the Newton--Okounkov body acts additively on a bigger set than $C_L(M)$. The next theorem shows, that this only happens in special situations.

\begin{Thm}\label{thm-reverse}
	Let $X$ be any irreducible projective variety, $L$ and $M$ ample $\mathbb{R}$-line bundles on $X$, and $Y_\bullet$ an admissible flag on $X$, such that $Y_1$ is a Cartier divisor on $X$. If $$\Delta_{Y_\bullet}(L+M)=\Delta_{Y_\bullet}(L)+\Delta_{Y_\bullet}(M),$$
	then there exists a convex cone $C\subseteq \partial \overline{\mathrm{Eff}}(X)$ in the boundary of the pseudo-effective cone $\partial\overline{\mathrm{Eff}}(X)=\overline{\mathrm{Eff}}(X)\setminus \mathrm{Big}(X)$, such that
	$$L,M\in \{\lambda\mathcal{O}_X(Y_1)+N~|~\lambda\in\mathbb{R}_{>0},~N\in C\}.$$
\end{Thm} 
In other words, if $\Delta_{Y_\bullet}$ acts additively on $L$ and $M$, then the projections of $L$ and $M$ by $\mathcal{O}_X(Y_1)$ to the boundary of the pseudo-effective cone lie in a convex subcone contained in the boundary. As the pseudo-effective cone is in general not even polyhedral, one can not expect that there are such subcones in the boundary of dimension higher than $1$ in general. If the convex subcone in the boundary is of dimension $1$, then $L$ and $M$ coincide modulo $\mathcal{O}_X(Y_1)$ up to a multiple. This is exactly the situation of Theorem \ref{thm_main}. Thus, in general situations the choice of the cone $C_{L}(M)$ in Theorem \ref{thm_main} is optimal as a subcone of $\mathrm{Amp}(X)$.

It should be emphasized that all of the previous research on the additivity of Newton--Okounkov bodies, as mentioned above, has been restricted to varieties whose pseudo-effective cone is polyhedral. Indeed, toric varieties, Bott--Samelson varieties (\cite[Theorem A.2]{And19}, \cite[Corollary 1.3.2]{BCHM10}), and spherical varieties \cite[Theorem 4.1.1]{Per14} are all Mori dream spaces, and it is known that the pseudo-effective cone of any Mori dream space is polyhedral \cite[Proposition 1.11 (2)]{HK00}. A polyhedral pseudo-effective cone is particularly significant in these results. For instance, the divisors $D_i$ in the Minkowski bases discussed above correspond exactly to the facets of the pseudo-effective cone. In light of Theorem 1.5, one cannot expect the existence of Minkowski bases for general varieties.

Finally, we want to give an application of Theorem \ref{thm_main} to inequalities of intersection numbers, which we will deduce from an analogous inequality of mixed volumes of convex bodies by Lehmann--Xiao \cite{LX17}.
\begin{Cor}\label{cor_intersection}
	Let $X$ be any irreducible projective variety of dimension $d$ and $L$, $M$, and $N$ nef $\mathbb{R}$-line bundles on $X$. Then we have the following inequality of intersection numbers
	$$L^d\cdot (M\cdot N^{d-1})\le d\cdot (M\cdot L^{d-1})\cdot (L\cdot N^{d-1}).$$
\end{Cor}
We remark that the much more general inequality
$$L^d\cdot (M^k\cdot N^{d-k})\le \tbinom{d}{k}\cdot (M^k\cdot L^{d-k})\cdot (L^k\cdot N^{d-k})$$
for any $0\le k\le d$ has recently been proved by Jiang and Li \cite{JL23} by a different strategy using so-called multipoint Okounkov bodies introduced by Trusiani \cite{Tru21}.
\section{Newton--Okounkov Bodies}\label{sec_okounkov}
	In this section, we recall the construction of Newton--Okounkov bodies and collect some relevant results. For further details, we refer to \cite{LM09}.
	Let $X$ be an irreducible projective variety of dimension $d$ over an algebraically closed field. The framework of \cite{LM09} applies in this general setting: it does not require smoothness or impose any condition on the characteristic of the base field. The only assumption is that the base field is uncountable. However, as noted in \cite[Footnote (5)]{LM09}, this condition is only used in a few instances, and is not needed for any of the results we cite in this section.
	 
	We choose a flag
	$$Y_{\bullet}:\quad X=Y_0\supsetneq Y_1\supsetneq \dots\supsetneq Y_{d-1}\supsetneq Y_d=\{p\}$$
	of irreducible subvarieties, which are regular at $p$. In the following we mean by $D$ always a Cartier divisor on $X$. We consider a valuation
	$$\nu_{Y_\bullet}\colon H^0(X,\mathcal{O}_X(D))\setminus\{0\}\to \mathbb{Z}^d,\quad s\to \nu_{Y_\bullet}(s)=(\nu_1(s),\dots,\nu_d(s)),$$
	which is defined as follows:\\
	We set $\nu_1=\nu_1(s)=\mathrm{ord}_{Y_1}(s)$, such that the restriction of $s$ induces a non-zero section $s_1\in H^0(Y_1,\mathcal{O}_{X}(D-\nu_1 Y_1)|_{Y_1})$. Inductively, we set $\nu_i=\nu_i(s)=\mathrm{ord}_{Y_i}(s_{i-1})$ and we write $s_i$ for the induced non-zero section 
	$$s_i\in H^0\left(Y_i, \mathcal{O}_{X}(D)|_{Y_i}\otimes \bigotimes_{k=1}^{i}\mathcal{O}_{Y_{k-1}}(-\nu_{k}Y_{k})|_{Y_i}\right).$$
	
	Now the Newton--Okounkov body of $D$ is defined by
	$$\Delta_{Y_\bullet}(D)=\mathrm{cch}\left(\bigcup_{m\ge 1}\tfrac{1}{m}\nu_{Y_\bullet}\left(H^0(X,\mathcal{O}_X(mD))\setminus\{0\}\right)\right)\subseteq \mathbb{R}^d,$$
	where $\mathrm{cch}$ stands for \emph{closed convex hull}. 
	The construction only depends on the linear equivalence class of $D$, so that we may simply write $\Delta_{Y_\bullet}(L)=\Delta_{Y_{\bullet}}(D)$ whenever $L\cong \mathcal{O}_X(D)$.
	As an example, if $X$ has dimension $1$, the Newton--Okounkov body of a big divisor $D$ on $X$ is just the line segment
	\begin{align*}
	\Delta_{Y_\bullet}(D)=[0,\deg D]\subseteq \mathbb{R},
	\end{align*}
	see \cite[Example 1.14]{LM09}.
	
	Let us recall some fundamental results about this construction from \cite{LM09}. First, for every big divisor it holds \cite[Theorem 2.3]{LM09}
	\begin{align}\label{equ_volume}
	\mathrm{vol}(\Delta_{Y_\bullet}(D))=\tfrac{1}{d!}\mathrm{vol}(D):=\lim_{m\to\infty}\frac{\dim H^0(X,\mathcal{O}_X(mD))}{m^d}.
	\end{align}
	Second, the Newton--Okounkov body $\Delta_{Y_\bullet}(D)$ only depends on the numerical equivalence class of $D$ if $D$ is big \cite[Proposition 4.1]{LM09}. For any integer $p>0$ and any big divisor $D$ the Newton--Okounkov body of $p\cdot D$
	\begin{align}\label{equ_factor}
	\Delta_{Y_\bullet}(p\cdot D)=p\Delta_{Y_\bullet}(D)
	\end{align}
	is just the scaling of the Newton--Okounkov body of $D$. 
	The proof of \cite[Proposition 4.1]{LM09} shows that Equation (\ref{equ_factor}) also holds if $D$ is any effective divisor.
	Hence, we may extend the definition of $\Delta_{Y_\bullet}$ to effective $\mathbb{Q}$-line bundles.
	It turns out that this extension is continuous on $\mathrm{Big}(X)$ \cite[Theorem B]{LM09}, such that we may define $\Delta_{Y_\bullet}$ also for big $\mathbb{R}$-divisors.
	
	In contrast to the compatibility with scalars in (\ref{equ_factor}), which by construction also holds for $p\in\mathbb{R}_{>0}$, the Newton--Okounkov body is in general not additive. But we obtain the weaker property that for any divisors $D_1$ and $D_2$ it always holds
	\begin{align}\label{equ_inclusion2}
	\Delta_{Y_\bullet}(D_1+D_2)\supseteq \Delta_{Y_\bullet}(D_1)+\Delta_{Y_\bullet}(D_2),
	\end{align}
	where the sum on the right-hand side denotes the Minkowski sum. This follows immediately from the construction, as for any sections $s_1\in H^0(X,\mathcal{O}_X(D_1))$ and $s_2\in H^0(X,\mathcal{O}_X(D_2))$ we obtain the section $s_1\otimes s_2\in H^0(X,\mathcal{O}_X(D_1+D_2))$, which induces the above inclusion.
	
	Finally, we recall a method to study the slices of Newton--Okounkov bodies.
	Let $E\subseteq X$ be an irreducible Cartier divisor. Further, let $M$ be a big line bundle satisfying $E\not \subseteq \mathbf{B}_+(M)$, where $\mathbf{B}_+(M)$ denotes the augmented base locus of $M$, see \cite[Section 2.4]{LM09} for its definition. It always holds $\mathbf{B}_+(M)=\emptyset$ if $M$ is ample.
	We set
	$$\mu(M;E)=\sup\{s>0~|~M-s\mathcal{O}_X(E)\in\mathrm{Big}(X)\}.$$
	We choose an admissible flag $Y_\bullet$, such that $Y_1=E$. By construction $\mu(M;E)$ coincides with the endpoint $\max\{t\ge 0~|~\Delta_{Y_\bullet}(M)_{\nu_1=t}\neq \emptyset\}$ of the Newton--Okounkov body $\Delta_{Y_\bullet}(M)$ after projecting to the first coordinate.
	For any big line bundle $N$ we define
	\begin{align}\label{equ_restriction}
	\Delta_{Y_\bullet|E}(N)=\mathrm{cch}\left(\bigcup_{m\ge 1}\tfrac{1}{m}\nu_{Y_{1,\bullet}}\left(\mathrm{Im}\left(H^0(X,N^{\otimes m})\to H^0(E,N^{\otimes m}|_E)\right)\setminus\{0\}\right)\right)
	\end{align}
	in $\mathbb{R}^{d-1}$, where the map denotes the restriction map and $Y_{1,\bullet}$ denotes the admissible flag $Y_1\supsetneq \dots\supsetneq Y_d$ on $E=Y_1$. 	
	Noting that 
	$$\Delta_{Y_\bullet|E}(pN)=p\Delta_{Y_\bullet|E}(N)$$
	for any $p\in\mathbb{Z}_{>0}$, we can canonically define $\Delta_{Y_\bullet|E}(N)$ for any big $\mathbb{Q}$-line bundle $N$.
	Then the slice of $\Delta_{\bullet}(M)$ at $\nu_1=t$ for any rational $0\le t<\mu(M;E)$ is given by
	\begin{align}\label{equ_slice}
	\Delta_{Y_\bullet}(M)_{\nu_1=t}=\{t\}\times\Delta_{Y_\bullet|E}(M-t\mathcal{O}_X(E)),
	\end{align}
	as proven in \cite[Theorem 4.26]{LM09}. 
	
	Note that the restriction map
	$$H^0(X,N^{\otimes m})\xrightarrow{\rho_m} H^0(E,N^{\otimes m}|_E)$$
	is surjective if $N$ is ample and $m$ is sufficiently large. Indeed, if we denote $\mathcal{I}_E$ for the ideal sheaf associated to the closed embedding $E\xrightarrow{\iota}X$, then we get a short exact sequence
	$$0\to \mathcal{I}_E\to \mathcal{O}_X\to \iota_*\mathcal{O}_E\to 0.$$
	Since tensoring with a locally free sheaf is exact, we get an exact sequence
	$$0\to \mathcal{I}_E\otimes_{\mathcal{O}_X}N^{\otimes m}\to N^{\otimes m}\to \iota_*\mathcal{O}_E\otimes_{\mathcal{O}_X}N^{\otimes m}\to 0.$$
	Note that $\iota_*\mathcal{O}_E\otimes_{\mathcal{O}_X}N^{\otimes m}\cong \iota_*\iota^*N^{\otimes m}\cong\iota_*(N^{\otimes m}|_E)$ by the projection formula. We obtain a long exact sequence of cohomology
	\begin{align*}
	0&\to H^0(X, \mathcal{I}_E\otimes_{\mathcal{O}_X}L^{\otimes m})\to H^0(X,L^{\otimes m})\\
	&\xrightarrow{\rho_m} H^0(X,L^{\otimes m}|_E)\to H^1(X,\mathcal{I}_E\otimes_{\mathcal{O}_X}L^{\otimes m})\to \dots.
	\end{align*}
	By Serre's vanishing theorem \cite[Theorem III.5.2]{Har77} there is an $m_0\in \mathbb{Z}$ depending on $X$, $L$, and $E$ such that $H^1(X,\mathcal{I}_E\otimes_{\mathcal{O}_X}L^{\otimes m})=0$ for all $m\ge m_0$. Hence, the restriction map $\rho_m$ is surjective for all sufficiently large $m$. In particular, we get
	\begin{align}\label{equ_nobody-restriction-ample}
		\Delta_{Y_\bullet|Y_1}(N)=\Delta_{Y_{1,\bullet}}(N|_{Y_1})
	\end{align}
	if $N$ is ample.
	\section{The Additivity of $\Delta_{Y_\bullet}$}\label{section_additivity}
	The aim of this section is to prove Theorem \ref{thm_main}.
		We will proceed by induction on the dimension $d=\dim X$ of $X$. If $d=1$ we get by (\ref{equ_volume})
		\begin{align*}
			&\mathrm{vol}(\Delta_{Y_{\bullet}}(N_1+N_2))=\deg(N_1+N_2)\\
			&=\deg N_1+\deg N_2=\mathrm{vol}(\Delta_{Y_{\bullet}}(N_1))+\mathrm{vol}(\Delta_{Y_{\bullet}}(N_2))
		\end{align*}	
		for any ample line bundles $N_1$ and $N_2$ on $X$. By $\mathbb{Z}$-linearity and by continuity this also holds for ample $\mathbb{R}$-line bundles.
		As the Newton--Okounkov bodies are line segments of the form $[0,c]$, we already obtain the equality in Theorem \ref{thm_main}.
		
		Now let $d>1$ and assume that the theorem holds for projective varieties of dimension $d-1$. Let $N_1,N_2\in C_{L}(M)$. Since we already have the inclusion (\ref{equ_inclusion2}), we only need to show
		\begin{align}\label{equ_inclusion-reversed}
			\Delta_{Y_\bullet}(N_1+N_2)\subseteq\Delta_{Y_\bullet}(N_1)+\Delta_{Y_\bullet}(N_2).
		\end{align}
		By continuity we can assume that $N_i=\lambda_iL+\mu_iM$ for some $\lambda_i,\mu_i\in \mathbb{Q}$ with $\mu_i>0$ for any $i=1,2$. We write $\mu=\mu(N_1+N_2,Y_1)$. Since both sides of (\ref{equ_inclusion-reversed}) are closed convex bodies, it is enough to show slice-wise
		\begin{align}\label{equ_inclusion-slicewise}
		\Delta_{Y_\bullet}(N_1+N_2)_{\nu_1=t}\subseteq\Delta_{Y_\bullet}(N_1)+\Delta_{Y_\bullet}(N_2)
		\end{align}
		for all $t\in \left(0,\mu\right)\cap\mathbb{Q}$, as $\left(0,\mu\right)\cap\mathbb{Q}$ is dense in the interval $\left[0,\mu\right]$ and the slices $\Delta_{Y_\bullet}(N_1+N_2)_{\nu_1=t}$ are empty for all $t\notin\left[0,\mu\right]$.
		By assumptions there is an $r\in\mathbb{Q}$ such that $r\mathcal{O}_X(Y_1)\equiv L$. After interchanging $N_1$ and $N_2$, we may assume
		$$\frac{r\lambda_1}{\mu_1}\le \frac{r\lambda_2}{\mu_2}.$$
		We set $t_0=r\lambda_2-\frac{\mu_2}{\mu_1} r\lambda_1\ge 0$.
		
		First we consider the case $t\ge t_0$. By (\ref{equ_slice}) we can compute the slice at $\nu_1=t$ by
		\begin{align}\label{equ_computation}
			\Delta_{Y_{\bullet}}(N_1+N_2)_{\nu_1=t}&=\Delta_{Y_\bullet}((1+\tfrac{\mu_2}{\mu_1})N_1+t_0\mathcal{O}_X(Y_1))_{\nu_1=t}\\
			&=		\{t\}\times\Delta_{Y_{\bullet}|Y_1}((1+\tfrac{\mu_2}{\mu_1})N_1-(t-t_0)\mathcal{O}_X(Y_1))\nonumber\\
			&=t_0 e_1+ \{t-t_0\}\times\Delta_{Y_{\bullet}|Y_{1}}((1+\tfrac{\mu_2}{\mu_1})N_1-(t-t_0) \mathcal{O}_X(Y_1))\nonumber\\
			&=t_0 e_1+\Delta_{Y_{\bullet}}((1+\tfrac{\mu_2}{\mu_1})N_1)_{\nu_1=t-t_0},\nonumber
		\end{align}
		where $e_1$ denotes the first vector of the standard basis of $\mathbb{R}^{d}$. Note that 
		$$e_1\in \Delta_{Y_{\bullet}}(\mathcal{O}_X(Y_1))$$
		since the canonical section of $\mathcal{O}_X(Y_1)$ vanishes of order $1$ at $Y_1$ and the line bundle $\mathcal{O}_X(Y_1)\otimes\mathcal{O}_{X}(-Y_1)\cong \mathcal{O}_{X}$ is trivial, such that
		$$H^0(Y_1,\mathcal{O}_X(Y_1)|_{Y_1}\otimes\mathcal{O}_{X}(-Y_1)|_{Y_1})\setminus\{0\}$$
		is non-empty and consists of non-zero constants, which have order $0$ restricted to any $Y_i$ for $i\ge 1$. Hence, the computation in (\ref{equ_computation}) implies
		$$\Delta_{Y_{\bullet}}(N_1+N_2)_{\nu_1=t}\subseteq \Delta_{Y_\bullet}(t_0\mathcal{O}_X(Y_1))+\Delta_{Y_\bullet}((1+\tfrac{\mu_2}{\mu_1})N_1).$$
		For the right-hand side we can compute
		\begin{align*}
			\Delta_{Y_\bullet}(t_0\mathcal{O}_X(Y_1))+\Delta_{Y_\bullet}((1+\tfrac{\mu_2}{\mu_1})N_1)
			&=\Delta_{Y_\bullet}(t_0\mathcal{O}_X(Y_1))+\Delta_{Y_\bullet}(\tfrac{\mu_2}{\mu_1}N_1)+\Delta_{Y_\bullet}(N_1)\\
			&\subseteq \Delta_{Y_\bullet}(t_0\mathcal{O}_X(Y_1)+\tfrac{\mu_2}{\mu_1}N_1)+\Delta_{Y_\bullet}(N_1)\\
			&=\Delta_{Y_\bullet}(N_2)+\Delta_{Y_\bullet}(N_1)
		\end{align*}
		Thus, we have shown inclusion (\ref{equ_inclusion-slicewise}) for $t\ge t_0$.
		
		Now we consider the case $t<t_0$. In particular, we have $t_0>0$ and hence, $r\neq 0$. We will apply the induction hypothesis to $Y_1$. Thus, we denote
		$$Y_{1,\bullet}\colon\quad Y_1\supsetneq Y_{2}\supsetneq\dots\supsetneq Y_{d-1}\supsetneq Y_d=\{p\}.$$
		for the restriction of the flag $Y_\bullet$ to $Y_1$.
		As $t<t_0$, the $\mathbb{Q}$-line bundle 
		$$N_1+N_2-\tfrac{t}{r}L=(\lambda_1+\lambda_2-\tfrac{t}{r})L+(\mu_1+\mu_2)M=(1+\tfrac{\mu_2}{\mu_1}\cdot\tfrac{t}{t_0})N_1+\tfrac{t_0-t}{t_0}N_2$$
		is ample by the ampleness of $N_1$ and $N_2$. It follows, that also $N_1+N_2-t\mathcal{O}_X(Y_1)$ is ample.
		Hence, we have
		\begin{align*}
			\Delta_{Y_{\bullet}|Y_{1}}(N_1+N_2-t\mathcal{O}_X(Y_1))&=\Delta_{Y_{1,\bullet}}(N_1|_{Y_1}+N_2|_{Y_1}-t\mathcal{O}_X(Y_1)|_{Y_1})\\
			&=\Delta_{Y_{1,\bullet}}((1+\tfrac{\mu_2}{\mu_1}\cdot\tfrac{t}{t_0})N_1|_{Y_1}+\tfrac{t_0-t}{t_0}N_2|_{Y_1}),
		\end{align*}
		where the first equality follows from Equation (\ref{equ_nobody-restriction-ample}).
		In a similar way, we also obtain
		$$\Delta_{Y_{\bullet}|Y_{1}}(N_2-t\mathcal{O}_X(Y_1))=\Delta_{Y_{1,\bullet}}(\tfrac{\mu_2}{\mu_1}\cdot\tfrac{t}{t_0}\cdot N_1|_{Y_1}+\tfrac{t_0-t}{t_0}N_2|_{Y_1})$$
		By the induction hypothesis we have
		$$\Delta_{Y_{1,\bullet}}((1+\tfrac{\mu_2}{\mu_1}\cdot\tfrac{t}{t_0})N_1|_{Y_1}+\tfrac{t_0-t}{t_0}N_2|_{Y_1})=\Delta_{Y_{1,\bullet}}(N_1|_{Y_1})+\Delta_{Y_{1,\bullet}}(\tfrac{\mu_2}{\mu_1}\cdot\tfrac{t}{t_0}\cdot N_1|_{Y_1}+\tfrac{t_0-t}{t_0}N_2|_{Y_1})$$
		Using the above observations and the slice formula (\ref{equ_slice}) we can finally compute
		\begin{align*}
			\Delta_{Y_\bullet}(N_1+N_2)_{\nu_1=t}=&\{t\}\times \Delta_{Y_\bullet|Y_1}(N_1+N_2-t\mathcal{O}_X(Y_1))\\
			=&\{t\}\times \Delta_{Y_{1,\bullet}}((1+\tfrac{\mu_2}{\mu_1}\cdot\tfrac{t}{t_0})N_1|_{Y_1}+\tfrac{t_0-t}{t_0}N_2|_{Y_1})\\
			=&\{t\}\times\left(\Delta_{Y_{1,\bullet}}(N_1|_{Y_1})+\Delta_{Y_{1,\bullet}}(\tfrac{\mu_2}{\mu_1}\cdot\tfrac{t}{t_0}\cdot N_1|_{Y_1}+\tfrac{t_0-t}{t_0}N_2|_{Y_1})\right)\\
			=&\{0\}\times\Delta_{Y_{1,\bullet}}(N_1|_{Y_1})+\{t\}\times \Delta_{Y_{1,\bullet}}(\tfrac{\mu_2}{\mu_1}\cdot\tfrac{t}{t_0}\cdot N_1|_{Y_1}+\tfrac{t_0-t}{t_0}N_2|_{Y_1})\\
			=&\{0\}\times\Delta_{Y_{\bullet}|Y_1}(N_1)+\{t\}\times\Delta_{Y_{\bullet}|Y_1}(N_2-t\mathcal{O}_X(Y_1))\\		
			=&\Delta_{Y_{\bullet}}(N_1)_{\nu_1=0}+\Delta_{Y_{\bullet}}(N_2)_{\nu_1=t},
		\end{align*}
		which implies the inclusion (\ref{equ_inclusion-slicewise}) for $t<t_0$. Thus, the proof of Theorem \ref{thm_main} is complete.
	\section{A Linear Map Compatible with Intersection Products}
	\label{sec_linear-map}
	In this section we deduce Corollary \ref{cor_main} from Theorem \ref{thm_main}. Let $L\in U$ be an ample $\mathbb{Q}$-line bundle. 
	We want to construct an admissible flag that corresponds to $L$. Let us show by induction that for every $0\le j\le d-1$ there is a flag of irreducible subvarieties $X=Y_0\supsetneq Y_1\supsetneq\dots\supsetneq Y_j$ such that for all $0\le i\le j-1$ it holds
	\begin{itemize}
		\item $Y_{i+1}$ is an irreducible Cartier divisor on $Y_{i}$,
		\item $Y_{i+1}\cap Y_{i,\mathrm{reg}}\subseteq Y_{i+1,\mathrm{reg}}$,
		\item there is an $r_i\in\mathbb{Q}$ with $r_i\mathcal{O}_{Y_i}(Y_{i+1})\cong L|_{Y_i}$ as $\mathbb{Q}$-line bundles, and
		\item $Y_j\cap X_{\mathrm{reg}}\neq \emptyset$.	
	\end{itemize}
	This is trivial for $j=0$ and we may assume that we have constructed such a flag for some $0\le j\le d-2$. We have to show that we can extend it by some $Y_{j+1}$ such that the properties are still valid for $0\le i\le j$ and $Y_{j+1}\cap X_{\mathrm{reg}}\neq \emptyset$.
	
	Since $L$ is ample, $L|_{Y_j}$ is ample, too. Thus, we can choose a positive integer $\widetilde{r}_{j}\in \mathbb{N}$ such that $\widetilde{r}_{j}L|_{Y_j}$ is very ample and hence, it induces an embedding $Y_{j}\to \mathbb{P}^N$. Since $\dim Y_{j}\ge 2$, we can use the following two versions of Bertini’s theorem: By \cite[Theorem 1.1]{FL81} any general
	hyperplane $H\subseteq \mathbb{P}^{N}$ intersects $Y_j$ in an irreducible Cartier divisor on $Y_j$. By \cite[Corollary 2]{CGM86} we have
	$(H\cap Y_j)_{\mathrm{reg}}\supseteq Y_{j,\mathrm{reg}}\cap H$ for any general hyperplane $H\subseteq\mathbb{P}^{N}$. Since every hyperplane $H\subseteq \mathbb{P}^{N}$ intersects $Y_j$, any general hyperplane $H\subseteq \mathbb{P}^{N}$ intersects the dense open subset $Y_j\cap X_{\mathrm{reg}}$ of $Y_j$.
	Hence, there is a hyperplane $H_1\subseteq \mathbb{P}^{N}$ such
	that $Y_{j+1}=H_1\cap Y_j$ is an irreducible Cartier divisor on $Y_j$, every regular point of $Y_j$ lying in $Y_{j+1}$ is also regular in $Y_{j+1}$, and $Y_{j+1}\cap X_{\mathrm{reg}}\neq \emptyset$. By construction we have $r_j\mathcal{O}_{Y_j}(Y_{j+1})\cong L|_{Y_j}$ for $r_j=\frac{1}{\widetilde{r}_j}\in \mathbb{Q}$. Thus, we have proved the induction hypothesis. 	
	
	To complete the flag $Y_0\supsetneq Y_1\supsetneq\dots\supsetneq Y_{d-1}$, we set $Y_d=\{p\}$ for any closed point $p\in Y_{d-1}\cap X_{\mathrm{reg}}$. By construction we get $p\in Y_{j,\mathrm{reg}}$ for all $0\le j\le d-1$ such that $Y_{\bullet}$ is an admissible flag and by the properties listed above it corresponds to $L$. A similar construction has been given by Anderson--Küronya--Lozovanu \cite[Proposition 4]{AKL14}.
		
	To continue with the prove of Corollary \ref{cor_main}, let $M\in U$ be linearly independent to $L$. Replacing $M$ by $M+ nL$ for some large enough $n\in \mathbb{N}$, we may assume that $M$ is also ample. Since $L$ and $M$ form a basis of $U$, we can write every $N\in U$ in the form $N=\lambda L+\mu M$ for some $\lambda,\mu\in \mathbb{Q}$. With this notation we define the map $\Delta\colon U\to\mathcal{K}_d^{\mathrm{vs}}$ by
	$$\Delta(N)=\lambda\Delta_{Y_{\bullet}}(L)+\mu\Delta_{Y_\bullet}(M).$$
	This map is clearly well-defined and linear. Since $L$ and $M$ are ample, Theorem \ref{thm_main} ensures that 
	\begin{align}\label{equ_linearmap}
	\Delta(\lambda L+\mu M)=\lambda\Delta_{Y_{\bullet}}(L)+\mu\Delta_{Y_\bullet}(M)=\Delta_{Y_\bullet}(\lambda L+\mu M)
	\end{align}
	for all $\lambda,\mu\ge 0$.
		
	Let us check the compatibility of $\Delta$ with intersection products.
	First, recall that the volume of an ample line bundle is its top degree intersection product, that means
	\begin{align*}
	\mathrm{vol}(\lambda L+\mu M)=(\lambda L+\mu M)^d
	\end{align*}
	for all $\lambda,\mu\ge 0$. Combining this with Equations (\ref{equ_volume}) and (\ref{equ_linearmap}) we obtain
	\begin{align}\label{equ_intersection-volume}
		\tfrac{1}{d!}(\lambda L+\mu M)^d=\mathrm{vol}(\Delta_{Y_\bullet}(\lambda L+\mu M))=\mathrm{vol}(\Delta(\lambda L+\mu M))
	\end{align}
	for all $\lambda,\mu\ge 0$.		
	 Further, by the polarization formula we get for the intersection number of line bundles $L_1,\dots,L_d$ as well as for the mixed volume of convex bodies $K_1,\dots,K_d$
	\begin{align*}
	L_1\cdot \ldots \cdot L_d&=\tfrac{1}{d!}\sum_{J\subseteq \{1,\dots,d\}}(-1)^{d-\#J}\left(\sum_{j\in J}L_j\right)^d,\\
	V(K_1,\dots,K_d)&=\tfrac{1}{d!}\sum_{J\subseteq \{1,\dots,d\}}(-1)^{d-\#J}\mathrm{vol}\left(\sum_{j\in J}K_j\right),
	\end{align*}
	where the second formula can also be taken as a definition for the mixed volume $V(K_1,\dots,K_d)$. If we apply Equation (\ref{equ_intersection-volume}) to every summand in the outer sum of these formulas, we obtain
	\begin{align}\label{equ_intersection_mixed}
	\tfrac{1}{d!}L^k\cdot M^{d-k}=V(\Delta(L)^k,\Delta(M)^{d-k})
	\end{align}
	for any $0\le k\le d$. Here, $V(K_1^k,K_2^{d-k})$ means $V(K_1,\dots,K_1,K_2,\dots,K_2)$ where the body $K_1$ occurs $k$-times and the body $K_2$ occurs $(d-k)$-times.
		
	To show the compatibility in general, we regard the intersection product and the mixed volume as linear maps
	$$\mathrm{int}\colon S^dU\to \mathbb{Q},\qquad V\colon S^d\mathcal{K}_d^{\mathrm{vs}}\to \mathbb{R},$$
	where $S^dW$ denotes the $d$-th symmetric power of any vector space $W$.
	As $S^d U$ is generated by symbolic elements of the form $L^kM^{d-k}$ for all $0\le k\le d$, the linear map $V\circ (S^d\Delta)\colon S^dU\to \mathbb{R}$ is completely determined by Equation (\ref{equ_intersection_mixed}). Thus, by linearity we indeed get
	$$\tfrac{1}{d!}(M_1\cdot\ldots\cdot M_d)=\tfrac{1}{d!}\mathrm{int}(M_1\cdots M_d)=V(\Delta(M_1),\dots,\Delta(M_d)).$$
		
	Next we show the injectivity of the map $\Delta$. It is enough to show that $\Delta(L)$ and $\Delta(M)$ are linearly independent in $\mathcal{K}_d^{\mathrm{vs}}$.
	Since $L$ and $M$ are linearly independent in $N^1(X)_\mathbb{R}$, it holds for their self-intersection numbers
	$$((L+M)^d)^{1/d}\neq (L^d)^{1/d}+(M^d)^{1/d},$$
	as worked out by Cutkosky \cite[Proposition 6.13]{Cut15} in this general situation.
	Using Equation (\ref{equ_intersection-volume}) we deduce a similar inequality for the volumes of the Newton--Okounkov bodies
	\begin{align}\label{equ_brunn_minkowski}
	\mathrm{vol}(\Delta(L)+\Delta(M))^{1/d}\neq \mathrm{vol}(\Delta(L))^{1/d}+\mathrm{vol}(\Delta(M))^{1/d}.
	\end{align}
	Note that it holds $\mathrm{vol}(\lambda K)^{1/d}=\lambda\mathrm{vol}(K)^{1/d}$ for all convex bodies $K$ and all $\lambda\in \mathbb{R}_{\ge 0}$. Hence, it follows from inequality (\ref{equ_brunn_minkowski}) that
	\begin{align}\label{equ_s1s2}
	s_1\Delta(L)\neq s_2\Delta(M)
	\end{align}
	for all $(s_1,s_2)\in(\mathbb{R}_{\ge 0})^2\setminus \{(0,0)\}$ and hence, also for $(s_1,s_2)\in(\mathbb{R}_{\le 0})^2\setminus \{(0,0)\}$. If $s_1,s_2\in \mathbb{R}$ have different signs, we also obtain (\ref{equ_s1s2}) as $s_1\Delta(L)-s_2\Delta(M)$ or $s_2\Delta(M)-s_1\Delta(L)$ is represented by a convex body of positive volume and hence non-zero. Thus, $\Delta(L)$ and $\Delta(M)$ are linearly independent.
		
	Finally, we consider ample line bundles $M_1,\dots,M_r\in U$ on $X$ and we want to show that the flag $Y_\bullet$ above can be chosen, such that $\Delta(M_j)=\Delta_{Y_\bullet}(M_j)$ for all $j\le r$. For this purpose, let $C\subseteq U$ denote the convex cone in $U$ consisting of all non-negative $\mathbb{Q}$-linear combinations of the $M_i$'s. After reordering we may assume that this cone is already generated by $M_1$ and $M_2$. We set $L=M_1$ and $M=M_2$ for the construction of the flag $Y_\bullet$ and the map $\Delta$ as above. If all $M_i$'s are multiples of each other, we choose $M$ instead to be any ample class in $U$ which is linearly independent of $L$. Then every $M_i$ can be represented in the form $M_i=\lambda_iL+\mu_iM$ for some $\lambda_i,\mu_i\in\mathbb{Q}_{\ge 0}$. As $Y_\bullet$ corresponds to $L$, we get by Theorem \ref{thm_main}
	$$\Delta(M_i)=\lambda_i\Delta_{Y_\bullet}(L)+\mu_i\Delta_{Y_\bullet}(M)=\Delta_{Y_\bullet}(\lambda_iL+\mu_iM)=\Delta_{Y_\bullet}(M_i)$$
	for all $i\le r$. This completes the proof Corollary \ref{cor_main}.
	
	\section{The Limits of the Additivity of $\Delta_{Y_\bullet}$}
	In this section we prove Theorem \ref{thm-reverse}. Let $L$ and $M$ be ample $\mathbb{R}$-line bundles on $X$ and $Y_\bullet$ an admissible flag on $X$, such that $Y_1$ is a Cartier divisor on $X$. We assume
	$$\Delta_{Y_\bullet}(L+M)=\Delta_{Y_\bullet}(L)+\Delta_{Y_\bullet}(M).$$
	We write $\mu_L=\mu(L,Y_1)$, $\mu_M=\mu(M,Y_1)$ and $\mu_{\Sigma}=\mu(L+M,Y_1)$ for the function $\mu$ defined in Section \ref{sec_okounkov}. The additivity of the Newton--Okounkov bodies assumed above implies
	$$\mu_{\Sigma}=\mu_{L}+\mu_{M}.$$
	Note, that the function $\mu$ was defined such that the $\mathbb{R}$-line bundles
	$$L^\circ=L-\mu_{L}\mathcal{O}_X(Y_1),\qquad M^\circ=M-\mu_{M}\mathcal{O}_X(Y_1)$$
	lie in the boundary $\partial\overline{\mathrm{Eff}}(X)$ of the pseudo-effective cone of $X$.
	To prove Theorem \ref{thm-reverse}, it is enough to show that the convex cone
	$$C=\{\lambda L^\circ+\mu M^\circ\in \overline{\mathrm{Eff}}(X)~|~\lambda,\mu\in\mathbb{R}_{\ge0}\}$$
	lies completely in the boundary $\partial\overline{\mathrm{Eff}}(X)$.
	Assume that there are $\lambda,\mu\in\mathbb{R}_{\ge 0}$ such that $\lambda L^\circ+\mu M^\circ\notin \partial\overline{\mathrm{Eff}}(X)$. Then there exists some $\epsilon>0$ such that 
	$$\lambda L^\circ+\mu M^\circ-\epsilon\mathcal{O}_X(Y_1)\in \overline{\mathrm{Eff}}(X).$$
	Without loss of generality we may assume $\lambda<\mu$, in particular $\mu>0$.
	Since $\overline{\mathrm{Eff}}(X)$ is a convex cone, we also get that 
	\begin{align*}
	\tfrac{1}{\mu}(\lambda L^\circ+\mu M^\circ-\epsilon\mathcal{O}_X(Y_1))+\tfrac{\mu-\lambda}{\mu}L^\circ=L+M-(\mu_{\Sigma}+\tfrac{\epsilon}{\mu})\mathcal{O}_X(Y_1)
	\end{align*}
	lies in $\overline{\mathrm{Eff}}(X)$ in contradiction to the definition of $\mu_{{\Sigma}}$. Thus, we get $C\subseteq \partial\overline{\mathrm{Eff}}(X)$ as desired.
	\section{An Inequality of Intersection Numbers}
	Finally, we give the proof of Corollary \ref{cor_intersection} in this section. It has been worked out by 	Lehmann--Xiao \cite[Theorem 5.9]{LX17}, that for all three convex bodies $K,L$ and $M$ in $\mathbb{R}^d$ we always have
	$$\mathrm{vol}(L)V(K^k,M^{d-k})\le\tbinom{d}{k}V(K^k,L^{d-k})V(L^k,M^{d-k})$$
	for any $0\le k\le d$. If we apply this for $k=1$ to the Newton--Okounkov bodies of three ample line bundles $L,M$ and $N$, we get
	\begin{align}\label{equ_inequality-mixed-okounkov}
	\tfrac{1}{d!}L^d  V(\Delta_{Y_\bullet}(M),\Delta_{Y_\bullet}(N)^{d-1})\le d V(\Delta_{Y_\bullet}(M),\Delta_{Y_\bullet}(L)^{d-1}) V(\Delta_{Y_\bullet}(L),\Delta_{Y_\bullet}(N)^{d-1})
	\end{align}
	for any admissible flag $Y_\bullet$ on $X$. To compare the mixed volumes of Newton--Okounkov bodies with the intersection numbers of ample line bundles, we prove the following lemma.
	\begin{Lem}\label{lem_mixed-volumes}
		Let $X$ be a projective variety of dimension $d$ and $L$ and $M$ any two ample line bundles on $X$. For any admissible flag $Y_\bullet$ it holds 
		$$V(\Delta_{Y_\bullet}(L),\Delta_{Y_\bullet}(M)^{d-1})\le\tfrac{1}{d!} (L\cdot M^{d-1}).$$
		If $Y_{\bullet}$ is corresponding to $L$ or to $M$, then the above inequality is an equality.
	\end{Lem}
	\begin{proof}
		Note, that the mixed volume is equal to the first derivative of the volume function at $t=0$ divided by $d$
		\begin{align}\label{equ_mixed-volume-derivative}
		V(\Delta_{Y_\bullet}(L),\Delta_{Y_\bullet}(M)^{d-1})=\frac{1}{d}\cdot\frac{d\mathrm{vol}(t\Delta_{Y_\bullet}(L)+\Delta_{Y_\bullet}(M))}{dt}|_{t=0}.
		\end{align}
		Here and in the following we always mean the derivative for which $t$ approximates $0$ in the positive rationals.
		For any $t\ge 0$ the inclusion (\ref{equ_inclusion2}) implies the inequality
		$$\mathrm{vol}(t\Delta_{Y_\bullet}(L)+\Delta_{Y_\bullet}(M))\le \mathrm{vol}(\Delta_{Y_\bullet}(tL+M))=\tfrac{1}{d!}(tL+M)^d,$$
		which is an equality at $t=0$. Thus the first derivative of the left-hand side has to be bounded by the first derivative of the right-hand side at $t=0$
		$$\frac{d\mathrm{vol}(t\Delta_{Y_\bullet}(L)+\Delta_{Y_\bullet}(M))}{dt}|_{t=0}\le \frac{1}{d!}\cdot\frac{d(tL+M)^d}{dt}|_{t=0}=\tfrac{d}{d!}(L\cdot M^{d-1}).$$
		Applying this inequality to Equation (\ref{equ_mixed-volume-derivative}) we get the inequality in the lemma. If $Y_\bullet$ corresponds to $L$ or $M$, we get by Theorem \ref{thm_main} an equality
		$$\mathrm{vol}(t\Delta_{Y_\bullet}(L)+\Delta_{Y_\bullet}(M))= \mathrm{vol}(\Delta_{Y_\bullet}(tL+M))=\tfrac{1}{d!}(tL+M)^d.$$
		Thus, both sides must have the same derivative at $t=0$ as functions in $t$, such that we obtain an equality 
		$$V(\Delta_{Y_\bullet}(L),\Delta_{Y_\bullet}(M)^{d-1})=\tfrac{1}{d!} (L\cdot M^{d-1}).$$
		This completes the proof of the lemma.
	\end{proof}
	We choose $Y_\bullet$ in inequality (\ref{equ_inequality-mixed-okounkov}) to be a flag corresponding to $M$. Such a flag can be constructed in the same way as described in Section \ref{sec_linear-map}. With this choice for $Y_\bullet$ we can apply Lemma \ref{lem_mixed-volumes} to both sides of inequality (\ref{equ_inequality-mixed-okounkov}). This yields
	$$L^d\cdot (M\cdot N^{d-1})\le d\cdot (M\cdot L^{d-1})\cdot(L\cdot N^{d-1}).$$
	By continuity of the intersection numbers, this inequality still holds true if $L,M$ and $N$ are any nef $\mathbb{R}$-line bundles on $X$. This finishes the proof of Corollary \ref{cor_intersection}.

\end{document}